\documentclass{amsart}
\usepackage{amssymb}
\usepackage{mathrsfs}
\usepackage{enumitem}

\newtheorem{theorem}{Theorem}[section]
\newtheorem{corollary}[theorem]{Corollary}
\newtheorem{fact}[theorem]{Fact}
\newtheorem{lemma}[theorem]{Lemma}

\theoremstyle{definition}
\newtheorem{definition}[theorem]{Definition}
\newtheorem{question}[theorem]{Question}

\newcommand{\rstr}{{\upharpoonright}}
\newcommand{\inj}{\preccurlyeq}
\newcommand{\fto}{\preccurlyeq_{\mathrm{fto}}}
\newcommand{\bfto}{\preccurlyeq_{\mathrm{bfto}}}

\newcommand{\ZF}{\mathsf{ZF}}
\newcommand{\id}{\mathrm{id}}

\DeclareMathOperator{\dom}{dom}
\DeclareMathOperator{\ran}{ran}
\DeclareMathOperator{\mov}{mov}

\DeclareMathOperator{\calb}{\mathcal{B}}
\DeclareMathOperator{\cals}{\mathcal{S}}
\DeclareMathOperator{\calsf}{\mathcal{S}_{fin}}
\DeclareMathOperator{\calp}{\mathcal{P}}
\DeclareMathOperator{\calvf}{\mathcal{V}_{F}}
\DeclareMathOperator{\fin}{fin}

\title{Boundedly finite-to-one functions}
\author{Xiao Hu}
\address{School of Philosophy\\
	Wuhan University\\
	No.~299 Bayi Road\\
	Wuhan\\
	Hubei Province 430072\\
	People's Republic of China}
\email{hu\_xiaoxiao@outlook.com}
\author{Guozhen Shen}
\address{School of Philosophy\\
	Wuhan University\\
	No.~299 Bayi Road\\
	Wuhan\\
	Hubei Province 430072\\
	People's Republic of China}
\email{shen\_guozhen@outlook.com}

\thanks{Shen was partially supported by NSFC No.~12101466.}


\begin{document}

\begin{abstract}
A function is boundedly finite-to-one if there is a natural number $k$ such that each point has at most $k$ inverse images. In this paper, we prove in $\mathsf{ZF}$ (i.e., the Zermelo--Fraenkel set theory without the axiom of choice) several results concerning this notion, among which are the following:
\begin{enumerate}
	\item For each infinite set $A$ and natural number $n$, there is no boundedly finite-to-one function from $\mathcal{S}(A)$ to $\mathcal{S}_{\leq n}(A)$, where $\mathcal{S}(A)$ is the set of all permutations of $A$ and $\mathcal{S}_{\leq n}(A)$ is the set of all permutations of $A$ moving at most $n$ points.
	\item For each infinite set $A$, there is no boundedly finite-to-one function from $\mathcal{B}(A)$ to $\mathrm{fin}(A)$, where $\mathcal{B}(A)$ is the set of all partitions of $A$ such that every block is finite and $\mathrm{fin}(A)$ is the set of all finite subsets of $A$.
\end{enumerate}
\end{abstract}
\maketitle
\section{Introduction}
In \cite{Cantor1891}, Cantor proves that for each infinite set $A$ there is no injection from $\mathcal{P}(A)$ to $A$. Under the axiom of choice, for infinite sets $A,B$, the existence of a finite-to-one function from $A$ to $B$ implies the existence of an injection from $A$ to $B$. However, the case without the axiom of choice is different. For example, in the second Fraenkel model (cf.~\cite[pp.~197--198]{Halbeisen2017}), there is an infinite set $A$ from which there is a two-to-one function but no injection to $\omega$. Cantor's theorem is generalized in $\ZF$ from injections to finite-to-one functions by Forster~\cite{Forster2003}, who proves that for each infinite set $A$ there is no finite-to-one function from $\mathcal{P}(A)$ to $A$.

A function is boundedly finite-to-one if there is a natural number $k$ such that each point has at most $k$ inverse images. For a set $A$ and a natural number $n$, $\cals(A)$ is the set of all permutations of $A$, $\cals_{\leq n}(A)$ is the set of permutations of $A$ moving at most $n$ points, $\cals_n(A)$ is the set of permutations of $A$ moving exactly $n$ points, $\calb(A)$ is the set of all partitions of $A$ such that every block is finite, and $\mathrm{fin}(A)$ is the set of all finite subsets of $A$.

The analog of Cantor's theorem for $\cals(A)$ is proved by Dawson and Howard~\cite{DawsonHoward1976}, which states that for each infinite set $A$, there is no injection from $\cals(A)$ to $A$. Sonpanow and Vejjajiva~\cite{SonVejja2018} generalize this result by proving that for each infinite set $A$ and natural number $n$, there is no injection from $\cals(A)$ to $A^n$. Unlike the case for $\mathcal{P}(A)$, Shen and Yuan~\cite{ShenYuan2020b} prove that the existence of an infinite set $A$ and a finite-to-one function from $\cals(A)$ to $A$ is relatively consistent with $\ZF$. In this paper, we generalize the result of Sonpanow and Vejjajiva by showing that there is no boundedly finite-to-one function from $\cals(A)$ to $A^n$. In fact, we prove a more general result, which states that there is no boundedly finite-to-one function from $\cals(A)$ to $\cals_{\leq n}(A)$.

Phansamdaeng and Vejjajiva \cite{PhanVejja2022} prove in $\mathsf{ZF}$ that for each infinite set $A$, there is no injection from $\calb(A)$ to $\fin(A)$. In \cite{Shen2024}, it is asked whether it is provable in $\ZF$ that there is no finite-to-one function from $\calb(A)$ to $A$. In this paper, we partially answer the question by showing that there is no boundedly finite-to-one function from $\calb(A)$ to $\fin(A)$. As a consequence, for each infinite set $A$ and natural number $n$, there is no boundedly finite-to-one function from $\calb(A)$ to $\cals_{\leq n}(A)$. This solves Question~5.1(2) of \cite{Shen2024} and generalizes Theorem~2.4 of \cite{SonVejja2023}.

We also settle some questions on $\cals_n(A)$ which are left open in \cite{SonVejja2023}. We show that for all natural numbers $m,n$, if $m\geq n+2$ or $m=n$, there is an injection from $\cals_n(A)$ to $\cals_m(A)$ for all infinite sets $A$, but in all other cases, it is even unprovable in $\ZF$ that there exists a finite-to-one function from $\cals_n(A)$ to $\cals_m(A)$.

\section{Preliminaries}
Throughout this paper, we shall work in $\mathsf{ZF}$. In this section we introduce briefly the terminology and notation we use, and present some results needed to prove our main theorems.

For a function~$f$ and a set $A$, we use $\dom(f)$ for the domain of~$f$, $\ran(f)$ for the range of~$f$,
$f[A]$ for the image of $A$ under~$f$, $f^{-1}[A]$ for the inverse image of $A$ under~$f$,
and $f{\upharpoonright}A$ for the restriction of $f$ to~$A$. For functions~$f,g$,
we use $g\circ f$ for the composition of $g$ and~$f$. For a set $A$, $|A|$ is the cardinality of $A$.

A function is \emph{finite-to-one} if each point has at most finitely many inverse images. A function is \emph{boundedly finite-to-one} if there is a natural number $k$ such that each point has at most $k$ inverse images. Clearly, compositions of (boundedly) finite-to-one functions are also (boundedly) finite-to-one functions.

\begin{definition}
	Let $A,B$ be any sets.
	\begin{enumerate}
		\item $A\inj B$ if there is an injection from $A$ to $B$.
		\item $A\fto B$ if there is a finite-to-one function from $A$ to $B$.
		\item $A\bfto B$ if there is a boundedly finite-to-one function from $A$ to $B$.
		\item $A\not\inj B$, $A\not\fto B$, $A\not\bfto B$ denote the negations of $A\inj B$, $A\fto B$, $A\bfto B$ respectively.
	\end{enumerate}
\end{definition}

For a permutation $s$ of $A$, we write $\mov(s)$ for the set $\{a\in A\mid s(a)\ne a\}$ (i.e., the elements of $A$ moved by $s$). For pairwise distinct elements $a_0,\dots,a_n$ of $A$, we write $(a_0;\dots;a_n)$ for the permutation of $A$ that moves $a_0$ to $a_1$, $a_1$ to $a_2$, \dots, $a_{n-1}$ to $a_n$, and $a_n$ to $a_0$, and fixes all other elements of $A$.

\begin{definition}
	Let $A$ be any set and let $n$ be any natural number.
	\begin{enumerate}
		\item $\fin(A)$ is the set of all finite subsets of $A$.
		\item $\calb(A)$ is the set of all partitions of $A$ such that every block is finite.
		\item $\cals(A)$ is the set of all permutations of $A$.
		\item $\calsf(A)=\{s\in\cals(A)\mid \mov(s)\text{ is finite}\}$.
		\item $\cals_{\leq n}(A)=\{s\in\cals(A)\mid |\mov(s)|\leq n\}$.
		\item $\cals_n(A)=\{s\in\cals(A)\mid |\mov(s)|=n\}$.
	\end{enumerate}
\end{definition}

We use the symbol $\calb$ to denote this notion just because $|\calb(n)|$ is the $n$-th Bell number $B_n$,
that is, the number of partitions of a set that has exactly $n$ elements. The Bell numbers satisfy Dobinski's formula
(see, for example, \cite{Rota1964}):
\[
B_n=\frac{1}{e}\sum_{i=0}^{\infty}\frac{i^n}{i!}.
\]

\begin{fact}\label{fact1}
	For each infinite set $A$ and natural number $n$, $\cals_{\leq n}(A)\bfto\fin(A)$.
\end{fact}
\begin{proof}
	The function that maps each $s\in\cals_{\leq n}(A)$ to $\mov(s)\in\fin(A)$ is clearly boundedly finite-to-one.
\end{proof}

\begin{lemma}\label{lemma2}
	For each infinite set $A$ and natural number $n\ne 0$, $A^n\inj\cals_{n+1}(A)$.
\end{lemma}
\begin{proof}
	See~\cite{Shen2023b}.
\end{proof}

\begin{lemma}\label{SY2020}
	For each infinite set $A$, if $\omega\inj\cals(A)$, then $\cals(A)\not\fto\calsf(A)$.
\end{lemma}
\begin{proof}
	See~\cite[Theorem~3.20]{ShenYuan2020}.
\end{proof}

\begin{lemma}\label{S2024}
	For each infinite set $A$, if $\omega\inj\calb(A)$, then $\calb(A)\not\fto\fin(A)$.
\end{lemma}
\begin{proof}
	See~\cite[Theorem~4.8]{Shen2024}.
\end{proof}

\section{Two theorems on boundedly finite-to-one functions}
\begin{theorem}\label{main1}
	For each infinite set $A$ and natural number $n$, $\cals(A)\not\bfto\cals_{\leq n}(A)$.
\end{theorem}
Before we prove the theorem, we introduce a notation $\rhd$. For each $X\subseteq A$ and $s\in\cals(A)$ that moves only finitely many points, $s\rhd X$ is the permutation of $X$ defined by $(s\rhd X)(x)=s^{i+1}(x)$, where $i$ is the least natural number such that $s^{i+1}(x)\in X$.
\begin{proof}
	Assume on the contrary that there are an infinite set $A$ and a natural number $n$ such that $\cals(A)\bfto\cals_{\leq n}(A)$. Let $f$ be a function from $\cals(A)$ to $\cals_{\leq n}(A)$ such that each point has at most $k$ inverse images, where $k$ is a fixed natural number. By Lemma~\ref{SY2020}, the existence of an injection from $\omega$ to $\cals(A)$ yields a contradiction. We construct such an injection $g$ by recursion as follows.
	
	Since
	\[
	\lim_{l\to\infty}\frac{2^l}{l^{2n}}=\infty,
	\]
	there is an $l_0\in\omega$ such that for all $l>l_0$ we have $\frac{2^l}{l^{2n}}>k(2n)^{2n}$, that is,
	\begin{equation}\label{eq01}
		k(2nl)^{2n}<2^l.
	\end{equation}
	Let $m_0=k(2nl_0)^{2n}$. Since $A$ is infinite, we can pick pairwise distinct elements $g(0),\dots,g(m_0)$ of $\cals(A)$.
	For $m>m_0$, we assume $g\rstr m$ is already defined, and now define $g(m)$.
	
	For each $i<m$, let $s_i=f(g(i))$, and write $B=\{s_i\mid i<m\}\subseteq\cals_{\leq n}(A)$. Clearly,
	\begin{equation}\label{eq02}
		m\leq k|B|.
	\end{equation}
	We define by recursion a sequence $\langle t_l\mid l\leq\log_2m\rangle$ of non-trivial permutations in $\cals_{\leq 2n}(A)$ with pairwise disjoint sets of non-fixed points. Let $l\leq\log_2m$ and assume that $t_i$ is already defined for all $i<l$. Let $C=\bigcup_{i<l}\mov(t_i)$. For all $i<l$, $|\mov(t_i)|\leq 2n$, so
	\begin{equation}\label{eq03}
		|C|\leq 2nl.
	\end{equation}
	Consider the following two cases:
	
	Case 1. There is a least $i<m$ such that, for some $x\in\mov(s_i)\setminus C,s_i(x)\in A\setminus C$. Then we define
	\[
		t_l=(s_i\rhd (A\setminus C))\cup\id_{C}.
	\]
	It is easy to see that $t_l$ is a non-trivial permutation in $\cals_{\leq n}(A)\subseteq\cals_{\leq 2n}(A)$ such that $\mov(t_l)\cap \mov(t_i)=\varnothing$ for all $i<l$.
	
	Case 2. Otherwise; that is,
	\begin{equation}\label{pr01}
		\text{for all $i<m$ and $x\in\mov(s_i)\setminus C$, $s_i(x)\in C$.}
	\end{equation}
	In this case, we claim that there are $i,j<m$ and $x\in C$ such that $s_i(x)$ and $s_j(x)$ are distinct elements of $A\setminus C$.
	
	Assume on the contrary that
	\begin{equation}\label{pr02}
		\text{for all }i,j<m\text{ and }x\in C,\text{ if }s_i(x),s_j(x)\in A\setminus C,\text{ then }s_i(x)=s_j(x).
	\end{equation}
	Consider the function $h$ on $B$ defined by
	\[
		h(s)=\langle s\rhd C,\{x\in C\mid s(x)\in A\setminus C\}\rangle.
	\]
	
	We prove that $h$ is injective as follows. Let $i,j<m$ be such that $h(s_i)=h(s_j)$. Take an arbitrary $z\in A$. We have to prove $s_i(z)=s_j(z)$. Consider the following four cases.
	\begin{enumerate}[label=\upshape(\roman*), leftmargin=*, widest=iii]
		\item If $z\in C$ and $s_i(z)\in C$, since $h(s_i)=h(s_j)$, $s_j(z)\in C$ and so $s_i(z)=(s_i\rhd C)(z)=(s_j\rhd C)(z)=s_j(z)$.
		\item If $z\in C$ and $s_i(z)\in A\setminus C$, since $h(s_i)=h(s_j)$, $s_j(z)\in A\setminus C$ and so $s_i(z)=s_j(z)$ by~\eqref{pr02}.
		\item Let $z\in\mov(s_i)\setminus C$. Then $s_i(z),s_i^{-1}(z)\in C$ by~\eqref{pr01}, and so $s_i(s_i^{-1}(z))=s_j(s_i^{-1}(z))$ by (ii) with $z$ replaced by $s_i^{-1}(z)$. Now $z=s_j(s_i^{-1}(z))=s_j(s_j^{-1}(z))$,
		which means $s_i^{-1}(z)=s_j^{-1}(z)$. Hence $z\in\mov(s_j)\setminus C$, and then
		$s_j(z)\in C$ by~\eqref{pr01}. Thus we have $s_i(z)=(s_i\rhd C)(s_i^{-1}(z))=(s_j\rhd C)(s_j^{-1}(z))=s_j(z)$.
		\item Since the proof of (iii) is symmetric, $z\notin\mov(s_i)\cup C$ implies $z\notin\mov(s_j)\cup C$, and then $s_i(z)=s_j(z)=z$.
	\end{enumerate}

	Therefore, $h$ is injective. Note that, for each $s\in B$, $s\rhd C\in\cals_{\leq n}(C)$ and
    $\{x\in C\mid s(x)\in A\setminus C\}\subseteq C\cap\mov(s)$. Since $|\cals_{\leq n}(C)|\leq\binom{|C|}{n}\cdot n!$, we have, by~\eqref{eq03},
	\[
		|B|\leq\binom{2nl}{n}\cdot n!\cdot\binom{2nl}{n}\leq(2nl)^{2n}.
	\]
	But then by~\eqref{eq02} we have
	\[
		k(2nl_0)^{2n}=m_0<m\leq k|B|\leq k(2nl)^{2n},
	\]
	so $l_0<l$, which means
	\[
		m\leq k(2nl)^{2n}<2^l
	\]
	by~\eqref{eq01}, contradicting $l\leq\log_2m$.
	
	Thus there are least $i,j<m$ with $i<j$ such that, for some $x\in C$, $s_i(x)$ and $s_j(x)$ are distinct elements of $A\setminus C$. Define
	\[
		t_l=((s_j\circ s_i^{-1})\rhd(A\setminus C))\cup\id_C.
	\]
	$t_l$ is clearly non-trivial, and since $\mov(t_l)\subseteq(\mov(s_i)\cup\mov(s_j))\cap (A\setminus C)$, it follows that $t_l\in\cals_{\leq 2n}(A)$ and $\mov(t_l)\cap \mov(t_i)=\varnothing$ for all $i<l$.
	
	Given the sequence $\langle t_l\mid l\leq\log_2m\rangle$, we can define a natural injection $F$ from $\calp(\lfloor\log_2m\rfloor+1)$ to $\cals(A)$ by
	\[
		F(a)=\bigcup_{i\in a}(t_i\rhd\mov(t_i))\cup\id_{A\setminus\bigcup_{i\in a}\mov(t_i)}.
	\]
	So we build at least $2^{\lfloor\log_2m\rfloor+1}>m$ many distinct permutations of $A$. Now there must be a least $a\subseteq\lfloor\log_2m\rfloor+1$ under the lexicographical ordering such that $F(a)\notin g[m]$. Take $g(m)=F(a)$ and we are done.
\end{proof}

\begin{corollary}
	For each infinite set $A$ and natural number $n$, $\cals(A)\not\bfto A^n$.
\end{corollary}
\begin{proof}
	By Lemma~\ref{lemma2} and Theorem~\ref{main1}.
\end{proof}

\begin{corollary}\label{co01}
	For each infinite set $A$, $\cals(A)\not\bfto A$.
\end{corollary}

It is shown in~\cite[Theorem~4.16]{ShenYuan2020b} that, in Corollary~\ref{co01}, $\not\bfto$ cannot be replaced by $\not\fto$.

\begin{theorem}\label{main2}
	For each infinite set $A$, $\calb(A)\not\bfto\fin(A)$.
\end{theorem}
\begin{proof}
	Assume on the contrary that there is an infinite set $A$ such that $\calb(A)\bfto\fin(A)$. Let $f$ be a function from $\calb(A)$ to $\fin(A)$ such that each point has at most $k$ inverse images, where $k$ is a fixed natural number. By Lemma~\ref{S2024}, the existence of an injection from $\omega$ to $\calb(A)$ yields a contradiction. We construct such an injection $g$ by recursion as follows.
	
	Since $A$ is infinite, we can pick pairwise distinct elements $g(0),\dots,g(72k^2)$ of $\calb(A)$.
	For $m>72k^2$, we assume $g\rstr m$ is already defined, and now define $g(m)$.
	
	Write $C=f[g[m]]\subseteq\fin(A)$. Clearly,
	\begin{equation}\label{eq11}
		m\leq k|C|.
	\end{equation}
	We can easily induce a well-ordering $<_C$ of $C$ via $f,g$ and the natural ordering of $m$. Let $\sim$ be the equivalence relation on $\bigcup C$ defined by
	\[
	x\sim y\quad\text{if and only if}\quad\forall c\in C(x\in c\leftrightarrow y\in c).
	\]
	Let $D$ be the quotient set of $\bigcup C$ by $\sim$. Then we define a well-ordering $<_D$ of $D$ by
	\[
	u<_Dv\quad\text{if and only if}\quad\exists c\in C(u\nsubseteq c\wedge v\subseteq c\wedge\forall b<_Cc(u\subseteq b\leftrightarrow v\subseteq b)).
	\]
	Let $l=|D|$. Since there is an obvious injection from $C$ into $\calp(D)$, $|C|\le2^l$. By \eqref{eq11},
	\begin{equation}\label{eq12}
		m\leq k2^l.
	\end{equation}
	
	We write $B_n$ for the $n$-th Bell number for each $n\in\omega$. Then $|\calb(D)|=B_l$. By the Dobinski formula,
	\begin{equation}\label{eq13}
		B_l
		=\frac{1}{e}\sum_{i=0}^{\infty}\frac{i^l}{i!}
		>\frac{1}{e}\cdot\frac{4^l}{4!}
		>\frac{4^l}{72}.
	\end{equation}
	Recall that $m>72k^2$, which, together with \eqref{eq12}, implies $72k^2<k2^l$, that is,
	\begin{equation}\label{eq14}
		72k<2^l.
	\end{equation}
	But then by \eqref{eq12}, \eqref{eq13} and \eqref{eq14} we have
	\begin{equation*}
		m\leq k2^l<\frac{4^l}{72}<B_l.
	\end{equation*}
	That is, $|\calb(D)|>m$.
	
	The lexicographical ordering of $\calp(D)$ induced by $<_D$ is a well-ordering since $\calp(D)$ is finite, which induces a well-ordering $R$ of $\calp(\calp(D))$ in the same way. Since $\calb(D)\subseteq\calp(\calp(D))$ and $|\calb(D)|>m$, there is an $R$-least $Q\in\calb(D)$ such that
	\[
	\{\textstyle\bigcup q\mid q\in Q\}\cup\{\{x\}\mid x\in A\setminus\bigcup C\}
	\]
	is a partition in $\calb(A)\setminus g[m]$. Define $g(m)$ to be such a partition and we are done.
\end{proof}

\begin{corollary}\label{co10}
	For each infinite set $A$ and natural number $n$, $\calb(A)\not\bfto\cals_{\leq n}(A)$.
\end{corollary}
\begin{proof}
	By Fact~\ref{fact1} and Theorem~\ref{main2}.
\end{proof}
Corollary~\ref{co10} is a generalization of  \cite[Theorem~2.4]{SonVejja2023}.

\begin{corollary}
	For each infinite set $A$ and natural number $n$, $\calb(A)\not\bfto A^n$.
\end{corollary}
\begin{proof}
	By Lemma~\ref{lemma2} and Corollary~\ref{co10}.
\end{proof}

\begin{corollary}
	For each infinite set $A$, $\calb(A)\not\bfto A$.
\end{corollary}
The question whether the existence of an infinite set $A$ such that $\calb(A)\fto A$ is relatively consistent with $\ZF$ is still open.

\section{Relationship between $\cals_m(A)$ and $\cals_n(A)$}
In \cite[Theorem~3.2]{NuVejja2022}, it is shown that $\cals_n(A)\inj\cals_{n+1}(A)$ is not provable in $\mathsf{ZF}$. Following on from that, Sonpanow and Vejjajiva prove in \cite[Theorem~2.7]{SonVejja2023} that for any infinite set $A$ and any natural numbers $m,n>1$ such that $m\geq2n$, $\cals_n(A)\inj\cals_m(A)$, and ask whether this is the best possible result. Below we negatively answer this question by providing an injection from $\cals_n(A)$ to $\cals_m(A)$ for every natural numbers $m,n$ with $m\geq n+2$. In addition, we prove that $\cals_n(A)\fto\cals_{n+1}(A)$ is not provable in $\mathsf{ZF}$, as a consequence of which, $\cals_n(A)\fto\cals_m(A)$ is not provable in $\ZF$ for every natural numbers $m,n$ with $n>m>0$.

The main idea of the proof of the following theorem is originally from~\cite[Lemma~(i)]{Truss1973} (cf.~also~\cite{Shen2023b}).
\begin{theorem}\label{main3}
	For each infinite set $A$ and natural numbers $m,n$ with $m\geq n+2$,
	$\cals_n(A)\inj\cals_m(A)$.
\end{theorem}
\begin{proof}
	Let $A$ be an infinite set and let $m,n$ be natural numbers with $m\geq n+2$.
	Since $2^{i}=1+\sum_{k<i}2^k$, we can pick $(m-n)(2^{n+1}-1)$ distinct elements of $A$ so that they are divided into
	$n+1$ sets $H_i$ ($i\leq n$) with
	\[
	H_i=\{a_{i,j}\mid j<m-n\}\cup\{b_{i,x}\mid x\in H_k\text{ for some }k<i\}.
	\]
	We construct an injection $f$ from $\cals_n(A)$ into $\cals_m(A)$ as follows.
	
	Let $s\in \cals_n(A)$. Let $i_s$ be the least $i\leq  n$ such that $\mov(s)\cap H_i=\varnothing$.
	There is such an $i$ because $|\mov(s)|=n$.
	Let $h_s$ be the involution of $A$ that swaps $x$ and $b_{i_s,x}$
	for all $x\in\mov(s)\cap\bigcup_{k<i_s}H_k$, and fixes all other elements of $A$.
	Let $t_s\in \cals_n(A)$ and $u_s\in \cals_{m-n}(A)$ be defined by
	\begin{align*}
		t_s & =h_s\circ s\circ h_s,\\
		u_s & =(a_{i_s,0};\dots;a_{i_s,m-n-1}).
	\end{align*}
	Now, we define
	\[
	f(s)=t_s\circ u_s.
	\]
	Clearly, $f(s)\in \cals_m(A)$.

\pagebreak
	
	We prove that $f$ is injective by showing that $s$ is uniquely determined by $f(s)$ in the following way.
	First, $i_s$ is the least $i\leq  n$ such that $f(s)$ moves some elements of~$H_i$.
	Second, $t_s$ is the permutation of $A$ that coincides with $f(s)$ except that it fixes all elements of $\{a_{i_s,j}\mid j<m-n\}$.
	Then, $h_s$ is the involution of $A$ that swaps $x$ and $b_{i_s,x}$ for all $x\in\bigcup_{k<i_s}H_k$
	such that $b_{i_s,x}\in\mov(t_s)$, and fixes all other elements of $A$. Finally,
	\[
	s=h_s\circ t_s\circ h_s.
	\]
	Hence, $f$ is an injection from $ \cals_n(A)$ into $ \cals_m(A)$.
\end{proof}

\begin{theorem}\label{main4}
	It is relatively consistent with $\ZF$ that, for some infinite set $A$, $\cals_n(A)\not\fto\cals_{n+1}(A)$ for all $n>1$.
\end{theorem}
\begin{proof}
	By the Jech--Sochor first embedding theorem (cf.~\cite[Theorem~17.2]{Halbeisen2017} or~\cite[Theorem~6.1]{Jech1973}), it suffices to find such an infinite set in the basic Fraenkel model $\calvf$ (cf.~\cite[pp.~195--196]{Halbeisen2017} or \cite[p.~48]{Jech1973}).
The set $A$ of atoms of $\calvf$ is denumerable,
and $x\in\calvf$ if and only if $x\subseteq\calvf$
and $x$ has a \emph{finite support}, that is, a set $E\in\fin(A)$ such that
every permutation of $A$ fixing $E$ pointwise also fixes $x$.
Below we prove that, in $\calvf$, $\cals_n(A)\not\fto\cals_{n+1}(A)$ for all $n>1$.

	Assume on the contrary that $f\in\calvf$ is a finite-to-one function from $\cals_n(A)$ to $\cals_{n+1}(A)$ and $E$ is a finite support of $f$. Let $s$ be any permutation in $\cals_n(A)$ such that $\mov(s)\subseteq A\setminus E$. We claim that $\mov(f(s))=\mov(s)\cup B$ for some $B\subseteq E$.
	
	Assume there is an $a\in\mov(s)\setminus\mov(f(s))$. For each $b\in A\setminus(E\cup\mov(f(s)))$, the transposition $\pi_b=(a;b)$ fixes $E\cup\mov(f(s))$ pointwise, and hence fixes $f$ and $f(s)$, which implies $f(\pi_b(s))=\pi_b(f(s))=f(s)$. So $f$ maps the infinite set $\{\pi_b(s)\mid b\in A\setminus(E\cup\mov(f(s)))\}$ to $f(s)$, contradicting that $f$ is finite-to-one. Thus $\mov(s)\subseteq\mov(f(s))$.
	
	Assume there is an $a\in\mov(f(s))\setminus(\mov(s)\cup E)$. Let $b\in A\setminus(E\cup\mov(f(s)))$. The transposition $\tau=(a;b)$ fixes $E\cup\mov(s)$ pointwise, and hence fixes $f$ and $s$. But $\tau(f(s))\ne f(s)$, a contradiction. Thus $\mov(f(s))\subseteq\mov(s)\cup E$.
	
	Since $|\mov(s)|=n$ and $|\mov(f(s))|=n+1$, there is exactly one point $e\in E$ such that $e\in\mov(f(s))\setminus\mov(s)$, that is, $\mov(f(s))=\mov(s)\cup\{e\}$. Let $d=f(s)(e)$. Then $d\in\mov(s)$. Since $s$, as a permutation, fixes $E$ pointwise, it follows that
	\[
		s(d)=s(f(s)(e))=f(s(s))(s(e))=f(s)(e)=d,
	\]
	contradicting $d\in\mov(s)$.
\end{proof}

\begin{corollary}
	It is relatively consistent with $\ZF$ that, for some infinite set $A$, $\cals_n(A)\not\fto\cals_m(A)$ for all $m,n\in\omega$ with $n>m>0$.
\end{corollary}
\begin{proof}
	By Theorem~\ref{main4}, it is relatively consistent with $\ZF$ that, for some infinite set $A$, $\cals_n(A)\not\fto\cals_{n+1}(A)$ for all $n>1$. For such an infinite set $A$, if $\cals_n(A)\fto\cals_m(A)$ for some $m,n\in\omega$ with $n>m>0$, since $n+1\geq m+2$, it follows that
	\[
		\cals_n(A)\fto\cals_m(A)\inj\cals_{n+1}(A)
	\]
	by Theorem~\ref{main3}, a contradiction.
\end{proof}

Given the previous results, we conclude that the relationship between $\cals_m(A)$ and $\cals_n(A)$ is as follows:

\begin{theorem}
	Let $A$ be any infinite set and let $m,n>1$ be any natural numbers.
	\begin{enumerate}[label=\upshape(\roman*), leftmargin=*, widest=ii]
		\item If $m\geq n+2$ or $m=n$, then $\cals_n(A)\inj\cals_m(A)$.
		\item If $m=n+1$ or $m<n$, then $\cals_n(A)\fto\cals_m(A)$ is not provable in $\mathsf{ZF}$.
	\end{enumerate}
\end{theorem}

\section{Open questions}
We conclude the paper with the following two open questions.

\begin{question}
Is it relatively consistent with $\ZF$ that there exists an infinite set $A$ such that $\calb(A)\fto A$?
\end{question}

\begin{question}\label{que00}
Does $\ZF$ prove that, for all sets $A,B$ and all positive integers $n,k$,
if there is a function from $n\times A$ to $n\times B$ such that each point has at most $k$ inverse images,
then there is a function from $A$ to $B$ such that each point has at most $k$ inverse images.
\end{question}

Note that the case $k=1$ of Question~\ref{que00} holds true and is known as the Bernstein--Lindenbaum division theorem (cf.~\cite{Hinkis2013}).

\subsection*{Acknowledgements}
We would like to give thanks to two anonymous referees for making useful suggestions.

\end{document}